\documentclass[reqno]{amsart}
\usepackage{amssymb}
\usepackage{amsfonts}
\usepackage{pifont}
\usepackage{stmaryrd}
\usepackage{bbm}
\usepackage{amsfonts}
\usepackage{amsmath}
\usepackage{amssymb}
\usepackage{mathrsfs}
\usepackage{accents}
\usepackage{amscd}
\allowdisplaybreaks
\newtheorem{theorem}{Theorem}[section]

\newtheorem{lemma}[theorem]{Lemma}

\theoremstyle{definition}

\theoremstyle{remark}

\theoremstyle{example}

\theoremstyle{fact}

\numberwithin{equation}{section}

\begin{document}
\title[Castelnuovo-Mumford regularity]{On the regularity of product of pure power complete intersections}
\author[Y.B. Gao]{Yubin Gao}
\address{School of Mathematics and Information Science, Shaanxi Normal
University, Xi'an 710062, PR China}
\email{gaoyb@snnu.edu.cn}
\subjclass[2010]{13D02}
\keywords{Castelnuovo-Mumford regularity, Complete intersection,
Product of ideals}

\begin{abstract}
 Let $I$ be a complete intersection in a polynomial ring over a field,
the Castelnuovo-Mumford regularity of $I^n$ is given by using an induction method. When $I$, $J$ and $K$ are three pure power
complete intersections, it is proved that $\mathrm{reg}(IJK)\leq \mathrm{reg}(I)+\mathrm{reg}(J)+\mathrm{reg}(K)$.
\end{abstract}

\maketitle
\date{}

\section{Introduction}
Let $S$ be a polynomial ring over a field $k$ and $\mathbbm{m}$ the maximal homogeneous ideal
of $S$. For a finitely generated graded $S$-module $M$, let
$a^i(M)=\max\{t\mid [H^i_\mathbbm{m}(M)]_t\neq 0\}$, and
$a^i(M)=-\infty$ if $H^i_\mathbbm{m}(M)=0$. The Castelnuovo-Mumford regularity of $M$ is defined as
\[\mathrm{reg}(M)=\max\limits_{i\geq 0}\{a^i(M)+i\}.\]
For a homogenous ideal $I$ of $S$, the highest degree of a generator of $IM$ is not more than the sum of the highest degree of a
generator of $M$ and that of a generator of $I$. So it is
natural to study whether $ \mathrm{reg}(IM)\leq \mathrm{reg}(I)+\mathrm{reg}(M)$ holds or not. When $ \mathrm{dim}(R/I)\leq 1$,
Conca and Herzog [2] proved that $ \mathrm{reg}(IM)\leq \mathrm{reg}(I)+\mathrm{reg}(M)$. However, even for
monomial ideals $I$ and $J$, it is not always true that $\mathrm{reg}(IJ)\leq\mathrm{reg}(I)+\mathrm{reg}(J)$.
Sturmfels [5] gave examples of monomial
ideals $I$ such that $ \mathrm{reg}(I^2)>2 \mathrm{reg}(I)$. If $I_1,\cdots,I_d$ are ideals generated by linear forms, Conca and Herzog [3] showed
that $\mathrm{reg}(I_1\cdots I_d)=d$, and they conjectured that
\[\mathrm{reg}(I_1\cdots I_d)\leq \mathrm{reg}(I_1)+\cdots+\mathrm{reg}(I_d)\]
for complete intersection ideals $I_1,\cdots,I_d$. When $d=2$, the statement is true by Theorem 3.1 in Chardin, Minh and Trung [1]; when
$d\geq 3$, this problem remains unsolved yet. Recently, Tang and Gong [6] shows that $\mathrm{reg}(I^n)\leq n \mathrm{reg}(I)$
for a monomial complete intersection $I$ and $n\geq 1$.

In this paper, when the minimal set of monomial generators of
a monomial complete intersection $I$ is $\{u_1,\cdots,u_s\}$ with $\deg(u_1)\geq\cdots\geq \deg(u_s)$, we show that
$\mathrm{reg}(I^n)=nu_1+\sum_{j=2}^su_j-(s-1)$. When $I$, $J$ and $K$ are three pure power monomial complete
intersections, it is proved that $\mathrm{reg}(IJK)\leq \mathrm{reg}(I)+\mathrm{reg}(J)+\mathrm{reg}(K)$.

Using the long exact sequence of local cohomology
 modules associated to a short exact sequence, one obtains the following lemma.
\begin{lemma} Let $0\rightarrow N \rightarrow M \rightarrow P \rightarrow 0$ be
a short exact sequence of finitely generated graded $S$-modules. Then\\
(i) $\mathrm{reg}(M)\leq \max\{\mathrm{reg}(N), \mathrm{reg}(P)\}$.\\
(ii) $\mathrm{reg}(P)\leq \max\{\mathrm{reg}(M), \mathrm{reg}(N)-1\}$.\\
(iii) $\mathrm{reg}(N)\leq \max\{\mathrm{reg}(M), \mathrm{reg}(P)+1\}$.\\
(iv) $\mathrm{reg}(P)=\mathrm{reg}(M)$ if $\mathrm{reg}(N)< \mathrm{reg}(M)$.\\
(v) $\mathrm{reg}(P)=\mathrm{reg}(N)-1$ if $\mathrm{reg}(M)<\mathrm{reg}(N)$.
\end{lemma}
Let $x$ be a linear form in $S$ and $I$ a homogeneous ideal, by the exact sequence
\[0\rightarrow S/(I:x^n)(-n)\rightarrow S/I\rightarrow S/(I,x^n)\rightarrow 0,\]
we have the following lemma which is a slight generalization of Lemma 2.1 in
Hoa and Trung [4].
\begin{lemma}
Let $x$ be a linear form and $I$ a homogeneous ideal in a polynomial ring $S$ over
a field. Then $\mathrm{reg}(I)\leq\max\{\mathrm{reg}(I,x^n), \mathrm{reg}(I:x^n)+n\}$ for all $n\geq 1$.
\end{lemma}
\begin{lemma}
Let $u$ be a homogeneous polynomial of degree $d$ which is $S/I$-regular for a homogeneous ideal $I$. Then
$\mathrm{reg}(I,u)=\mathrm{reg}(I)+d-1$.
\end{lemma}
\begin{proof}
Since $u$ is $S/I$-regular, there is a short exact sequence of graded $S$-modules
\[0\rightarrow S/I(-d)\xrightarrow{.u}S/I\rightarrow S/(I,u)\rightarrow 0.\]
By Lemma 1.1,
\[\begin{array}{ll}
\mathrm{reg}(S/(I,u))\leq \max\{\mathrm{reg}(S/I(-d))-1,\mathrm{reg}(S/I)\}
=\mathrm{reg}(S/I)+d-1,\\
\mathrm{reg}(S/I(-d))\leq\max\{\mathrm{reg}(S/(I,u))+1,\mathrm{reg}(S/I)\}.
\end{array}\]
Since $\mathrm{reg}(S/I(-d))=\mathrm{reg}(S/I)+d$ and $d\geq 1$, we must have
$\mathrm{reg}(S/I)+d\leq \mathrm{reg}(S/(I,u))+1$. So
$\mathrm{reg}(S/(I,u))=\mathrm{reg}(S/I)+d-1$, that is, $\mathrm{reg}(I,u)=\mathrm{reg}(I)+d-1$.
\end{proof}

\section{Regularity of powers of monomial complete intersections}
Let $S$ be a polynomial ring over a field. An ideal $I=(u_1,\cdots,u_s)$ is called a complete
intersection if $u_1,\cdots,u_s$ is a regular sequence of $S$.
\begin{theorem}
Let $I$ be a complete intersection of $S$ with  minimal set of monomial generators
$\{u_1,\cdots,u_s\}$ such that $\mathrm{deg}(u_1)\geq\cdots\geq deg(u_s)$. Then for all $n\geq 1$,
\[\mathrm{reg}(I^n)=nu_1+\sum_{j=2}^su_j-(s-1).\]
\end{theorem}
\begin{proof}
We use induction on $s$ and $n$. For $s=1$ and any $n$, the
formula is obviously right; if $n=1$, the formula is already well known to be true for any $s$.
Now assume $s\geq 2$ and $n\geq 2$, let $J=(u_2,\cdots,u_s)$ and $u=u_1$ for convenience. Then
$I^n=(u,J)I^{n-1}=uI^{n-1}+J(u,J)^{n-1}=uI^{n-1}+J^n$ and there is  a short exact sequence
\[0\rightarrow uI^{n-1}\cap J^n\rightarrow uI^{n-1}\oplus J^n\rightarrow I^n\rightarrow 0.\]
Using induction hypothesis, we have
\[\begin{array}{ll}
\mathrm{reg}(uI^{n-1})=\deg(u)+\mathrm{reg}(I^{n-1})=n\deg(u)+\sum_{j=2}^s\deg(u_j)-(s-1).\\
\mathrm{reg}(J^n)=n\deg(u_2)+\sum_{j=3}^s\deg(u_j)-(s-2).
\end{array}\]
Since $\deg(u)\geq \deg(u_2)$, we get that $\mathrm{reg}(uI^{n-1})\geq \mathrm{reg}(J^n)$, so
\[\alpha=\mathrm{reg}(uI^{n-1}\oplus J^n)=\max\{\mathrm{reg}(uI^{n-1}),
\mathrm{reg}(J^n)\}=n\mathrm{deg}(u)+\sum_{j=2}^s\deg(u_j)-(s-1).\]
Note that $u$ is $S/J^n$-regular, so $uI^{n-1}\cap J^n=u[I^{n-1}\cap (J^n:u)]=u(I^{n-1}\cap J^n)
=uJ^n$. By induction on the cardinality of the minimal set of monomial generators, we have
\[\beta=\mathrm{reg}(uI^{n-1}\cap J^n)=\mathrm{reg}(uJ^n)=\deg(u)+n\deg(u_2)+\sum_{j=3}^s\deg(u_j)-(s-2),\]
then $\beta-\alpha=(n-1)(\deg(u_2)-\deg(u))+1$, note that we assume $n\geq 2$ here.
If $\deg(u)=\deg(u_2)$, then $\alpha<\beta$, and by Lemma 1.1$(v)$, we have
$\mathrm{reg}(I^n)=\beta-1=\alpha$; If $\deg(u)\geq\deg(u_2)+2$, then $\beta< \alpha$ and
$\mathrm{reg}(I^n)=\alpha$ by Lemma 1.1$(iv)$; If $\deg(u)=\deg(u_2)+1$ and $n>2$, then
 $\beta< \alpha$ and $\mathrm{reg}(I^n)=\alpha$.

 Now only one case remains to be considered, that is, $n=2$ and $\deg(u)=\deg(u_2)+1$.
 Set $K=(u,\hat{u_2},u_3,\cdots,u_s)$, where $\hat{u_2}$ denotes that $u_2$ is removed
 from the set of generators. We have $I^2=(u_2,K)^2=u_2I+K^2$ and there is an exact sequence
 \[0\rightarrow u_2I\cap K^2\rightarrow u_2I\oplus K^2\rightarrow I^2\rightarrow 0.\]
As above, we have $u_2I\cap K^2=u_2K^2$. Using induction hypothesis, we have
\[\begin{array}{ll}
\mathrm{reg}(u_2K^2)=\deg(u_2)+2\deg(u)+\sum_{j=3}^s\deg(u_j)-(s-2),\\
\mathrm{reg}(u_2I)=\deg(u_2)+\deg(u)+\sum_{j=2}^s\deg(u_j)-(s-1),\\
\mathrm{reg}(K^2)=2\deg(u)+\sum_{j=3}^s\deg(u_j)-(s-2).
\end{array}\]
Since $\mathrm{reg}(u_2K^2)$ is greater than both $\mathrm{reg}(u_2I)$ and $\mathrm{reg}(K^2)$, by Lemma 1.1(v), we
have $\mathrm{reg}(I^2)=\mathrm{reg}(u_2K^2)-1=\alpha$. Now the theorem is proved.
\end{proof}

\section{Regularity of product of three pure power complete intersections}
An ideal $I$ of $S$ is a pure power complete intersection if the elements of the minimal
set of monomial generators of $I$ are all powers of only one variable. The main techniques
we use are from Lemma 1.2.
\begin{lemma}
Let $I$, $J$ and $K$ be three pure power complete intersections. Then
\[\mathrm{reg}(IJ+IK+JK)\leq \mathrm{reg}(I)+\mathrm{reg}(J)+\mathrm{reg}(K)-1.\]
\end{lemma}
\begin{proof}
We use induction on $l_1+l_2+l_3$, where $l_1, l_2,$ and $l_3$ are the cardinality of the minimal set
of generators of $I$, $J$ and $K$ respectively. If $I$, $J$ and $K$ are all generated by one variable,
say $I=(x^l)$, $J=(y^m)$ and $K=(z^n)$ with $l\geq m\geq n$ and $x\neq y\neq z$.
Then $IJ+IK+JK=(x^ly^m,x^lz^n,y^mz^n)$. By Lemma 1.2,
\begin{align*}
\mathrm{reg}((IJ,IK,JK))\leq&\max\{\mathrm{reg}((IJ,IK,JK,z^n)),\mathrm{reg}(((IJ,IK,JK):z^n))+n\}\\
=&\max\{\mathrm{reg}((x^ly^m,z^n)), \mathrm{reg}((x^l,y^m))+n\}\\
=&l+m+n-1=\mathrm{reg}(I)+\mathrm{reg}(J)+\mathrm{reg}(K)-1.
\end{align*}
The cases that $x=y$ or $x=y=z$ can be considered similarly.

\ding {172} If $I=(I_1, x^m)$ and $x$ is a non-zero-divisor on $S/I_1$, $S/J$ and $S/K$, i.e.,
there is no power of $x$ in the minimal set of monomial generators of $I_1$, $J$ and $K$.
Then
\begin{align*}
(IJ,IK,JK)=&(I_1,x^m)J+(I_1,x^m)K+JK\\
=&I_1J+I_1K+JK+x^mJ+x^mK.
\end{align*}
By Lemma 1.2, we have
\begin{align*}
\mathrm{reg}((IJ,IK,JK))\leq&\max\{\mathrm{reg}((IJ,IK,JK,x^m)), \mathrm{reg}((IJ,IK,JK):x^m)+m\}\\
=&\max\{\mathrm{reg}((I_1J,I_1K,JK,x^m)),\mathrm{reg}((J,K))+m\}.
\end{align*}
Note that $x$ is $S/(I_1J,I_1K,JK)$-regular, then by Lemma 1.3 and induction hypothesis,
we have $\mathrm{reg}((I_1J,I_1K,JK,x^m))=\mathrm{reg}((I_1J,I_1K,JK))+m-1\leq \mathrm{reg}(I_1)+\mathrm{reg}(J)+\mathrm{reg}(K)+m-2
=\mathrm{reg}(I)+\mathrm{reg}(J)+\mathrm{reg}(K)-1$ because of $\mathrm{reg}(I)=\mathrm{reg}(I_1)+m-1$. By Corollary 3.2 in Herzog [3],
$\mathrm{reg}((J,K))+m\leq \mathrm{reg}(J)+\mathrm{reg}(K)+m-1$. So in this case,
$\mathrm{reg}((IJ,IK,JK))\leq \mathrm{reg}(I)+\mathrm{reg}(J)+\mathrm{reg}(K)-1$.

\ding {173} If $I=(I_1,x^m), J=(J_1,x^n)$ with $m\geq n$ and $x$ is $S/K$-regular. Then
\begin{align*}
(IJ,IK,JK)=&(I_1,x^m)(J_1,x^n)+(I_1,x^m)K+(J_1,x^n)K\\
=&I_1J_1+I_1K+J_1K+x^nI_1+x^mJ_1+x^nK+(x^{m+n}).
\end{align*}
By using Lemma 1.2 twice, we have
\begin{align*}
\mathrm{reg}((IJ,IK,JK))\leq&\max\{\mathrm{reg}((I_1J_1,I_1K,J_1K,x^n)), \mathrm{reg}((I_1,K,x^{m-n}J_1,x^m))+n\}\\
\leq&\max\{\mathrm{reg}((I_1J_1,I_1K,J_1K,x^n)),\mathrm{reg}((I_1,K,x^{m-n}))+n,\\
& \mathrm{reg}((I_1,J_1,K,x^n))+m\}.
\end{align*}
By the induction hypothesis and the fact that $x^n$ is $S/(I_1J_1,I_1K,J_1K)$-regular,
\begin{equation*}
\mathrm{reg}((I_1J_1,I_1K,J_1K,x^n))\leq \mathrm{reg}(I_1)+\mathrm{reg}(J_1)+\mathrm{reg}(K)+n-2.
\end{equation*}
By Corollary 3.2 in Herzog [3], we have
\begin{align*}
\mathrm{reg}((I_1,K,x^{m-n}))+n\leq&\mathrm{reg}(I_1)+\mathrm{reg}(K)+m-2;\\
\mathrm{reg}((I_1,J_1,K,x^n))+m\leq&\mathrm{reg}(I_1)+\mathrm{reg}(J_1)+\mathrm{reg}(K)+m+n-3.
\end{align*}
By Lemma 1.3, $\mathrm{reg}(I)=\mathrm{reg}(I_1)+m-1$ and $\mathrm{reg}J)=\mathrm{reg}(J_1)+n-1$. So we have
$\mathrm{reg}((IJ,IK,JK))\leq \mathrm{reg}(I)+\mathrm{reg}(J)+\mathrm{reg}(K)-1$ in this case from the above inequations.

\ding {174} If $I=(I_1,x^m), J=(J_1,x^n)$ and $K=(K_1,x^s)$ with $m\geq n\geq s\geq 1$. Then
we have
\[(IJ,IK,JK)=(I_1J_1,I_1K_1,J_1K_1,x^sI_1,x^sJ_1,x^nK_1,x^{n+s}).\]
As the above two cases, we have
\begin{align*}
\mathrm{reg}((IJ,IK,JK))\leq&\max\{\mathrm{reg}((I_1J_1,I_1K_1,J_1K_1,x^s)), \mathrm{reg}((I_1,J_1,x^{n-s}K_1,x^n))+s\}\\
\leq& \max\{\mathrm{reg}((I_1J_1,I_1K_1,J_1K_1,x^s)), \mathrm{reg}((I_1,J_1,x^{n-s}))+s,\\
&\mathrm{reg}((I_1,J_1,K_1,x^s))+n\}.
\end{align*}
By the induction hypothesis and Lemma 1.3,
\[\mathrm{reg}((I_1J_1,I_1K_1,J_1K_1,x^s))\leq \mathrm{reg}(I_1)+\mathrm{reg}(J_1)+\mathrm{reg}(K_1)+s-2.\]
By Corollary 3.2 in Herzog [3],
\begin{align*}
\mathrm{reg}((I_1,J_1,x^{n-s}))+s\leq& \mathrm{reg}(I_1)+\mathrm{reg}(J_1)+n-2,\\
\mathrm{reg}((I_1,J_1,K_1,x^s))+n\leq& \mathrm{reg}(I_1)+\mathrm{reg}(J_1)+\mathrm{reg}(K_1)+s+n-3.
\end{align*}
By Lemma 1.3, $\mathrm{reg}(I)+\mathrm{reg}(J)+\mathrm{reg}(K)-1=\mathrm{reg}(I_1)+\mathrm{reg}(J_1)+\mathrm{reg}(K_1)+m+n+s-4$. So in this case, we also
have $\mathrm{reg}((IJ,IK,JK))\leq \mathrm{reg}(I)+\mathrm{reg}(J)+\mathrm{reg}(K)-1$. Now the lemma is proved.

\end{proof}

The following theorem answers the question (1.1) raised in Conca and Herzog [2] in the case
that $d=3$ and $I_i$ are all pure power complete intersections.
\begin{theorem}
Let $I$, $J$ and $K$ be three pure power complete intersections in a polynomial ring $S$ over
a field $k$. Then
\[\mathrm{reg}(IJK)\leq \mathrm{reg}(I)+\mathrm{reg}(J)+\mathrm{reg}(K).\]
\end{theorem}
\begin{proof}
We use induction on $l_1+l_2+l_3$, where $l_1, l_2,$ and $l_3$ are the cardinality of the minimal set
of generators of $I$, $J$ and $K$ respectively. If $l_1=l_2=l_3=1$, the assertion of the theorem is
clearly right since $\mathrm{reg}(IJK)=l_1+l_2+l_3$.

\ding {172} If one variable $x$ of $S$ appears only in the minimal monomial generators of $I$, not in
 those of $J$ and $K$. Set $I=(I_1, x^m)$ with $m\geq 1$, then
 $IJK=I_1JK+x^mJK$ and $x^m$ is $S/I_1JK$-regular. By Lemma 1.2 and Lemma 1.3,
\begin{align*}
\mathrm{reg}(IJK)\leq&\max\{\mathrm{reg}((I_1JK,x^m)),\mathrm{reg}((I_1JK,x^mJK):x^m)+m\}\\
=&\max\{\mathrm{reg}(I_1JK)+m-1,\mathrm{reg}(JK)+m\}.
\end{align*}
By induction hypothesis, we have
\[\mathrm{reg}(I_1JK)+m-1\leq \mathrm{reg}(I_1)+\mathrm{reg}(J)+\mathrm{reg}(K)+m-1=\mathrm{reg}(I)+\mathrm{reg}(J)+\mathrm{reg}(K).\]
By Lemma 3.1 in Chardin, Minh and Trung [1], $\mathrm{reg}(JK)+m\leq \mathrm{reg}(J)+\mathrm{reg}(K)+m\leq \mathrm{reg}(I)+\mathrm{reg}(J)+\mathrm{reg}(K)$. So the conclusion
of the theorem is true in this case.

\ding {173} If one variable $x$ appears in the minimal monomial generators of $I$ and $J$, not in those
of $K$, set $I=(I_1,x^m)$ and $J=(J_1,x^n)$ with $m\geq n$. Then
\[IJK=(I_1J_1K, x^nI_1K, x^mJ_1K, x^{m+n}K).\]
By Lemma 1.2,
\begin{align*}
\mathrm{reg}(IJK)\leq& \max\{\mathrm{reg}((IJK,x^m)), \mathrm{reg}((IJK:x^m))+m\},\\
=&\max\{\mathrm{reg}((I_1J_1K,x^nI_1K,x^m)), \mathrm{reg}((I_1K,J_1K,x^nK))+m \} ,
\end{align*}
For the last two terms above, we have
\begin{align*}
\mathrm{reg}((I_1J_1K,x^nI_1K,x^m))\leq&\max\{\mathrm{reg}((I_1J_1K,x^n)),\mathrm{reg}((I_1K,x^{m-n}))+n\},\\
\mathrm{reg}((I_1K,J_1K,x^nK))+m\leq&\max\{\mathrm{reg}((I_1K,J_1K,x^n))+m,\mathrm{reg}(K)+m+n\}.
\end{align*}
By induction hypothesis and the fact that $x$ dose not appear in the
minimal set of monomial generators of $I_1$, $J_1$ and $K$,
\begin{align*}
\mathrm{reg}((I_1J_1K,x^n))\leq& \mathrm{reg}(I_1)+\mathrm{reg}(J_1)+\mathrm{reg}(K)+n-1\\
=&\mathrm{reg}(I_1)+\mathrm{reg}(J)+\mathrm{reg}(K),
\end{align*}
and
\begin{align*}
\mathrm{reg}((I_1K,x^{m-n}))+n=&\mathrm{reg}(I_1K)+m-1\\
\leq& \mathrm{reg}(I_1)+\mathrm{reg}(K)+m-1\\
=& \mathrm{reg}(I)+\mathrm{reg}(K).
\end{align*}
Note that $I_1+J_1$ is still a pure power monomial complete intersection, by
Corollary 3.2 in Herzog [3], we have
\begin{align*}
\mathrm{reg}((I_1K,J_1K,x^n))+m=&\mathrm{reg}((I_1,J_1)K)+m+n-1\\
\leq& \mathrm{reg}(I_1,J_1)+\mathrm{reg}(K)+m+n-1\\
\leq& \mathrm{reg}(I_1)+\mathrm{reg}(J_1)+\mathrm{reg}(K)+m+n-2\\
=& \mathrm{reg}(I)+\mathrm{reg}(J)+\mathrm{reg}(K).
\end{align*}
Now the theorem is proved in this case by the above inequations.

\ding {174} If the variable $x$ appears in the minimal monomial generators of $I$, $J$,
and $K$, set $I=(I_1,x^m)$, $J=(J_1,x^n)$ and $K=(K_1,x^s)$ with $m\geq n\geq s\geq 1$.
Then
\[IJK=(I_1J_1K_1,x^sI_1J_1,x^nI_1K_1,x^mJ_1K_1,x^{n+s}I_1,x^{m+s}J_1,x^{m+n}K_1,x^{m+n+s}).\]
We assume that $m\leq n+s$ first. By Lemma 1.2,
\begin{align}
\mathrm{reg}(IJK)\leq& \max\{\mathrm{reg}((IJK,x^{n+s})), \mathrm{reg}((IJK:x^{n+s}))+n+s\},\nonumber\\
=&\max\{\mathrm{reg}((I_1J_1K_1,x^sI_1J_1,x^nI_1K_1,x^mJ_1K_1,x^{n+s})),\\
&\mathrm{reg}((I_1,J_1K_1,x^{m-n}J_1,x^{m-s}K_1,x^m ))+n+s\}.
\end{align}
For the term in (3.1), we have
\begin{align*}
&\mathrm{reg}((I_1J_1K_1,x^sI_1J_1,x^nI_1K_1,x^mJ_1K_1,x^{n+s}))\\
\leq&\max\{\mathrm{reg}((I_1J_1K_1,x^sI_1J_1,x^nI_1K_1,x^m)),
\mathrm{reg}((I_1J_1,I_1K_1,J_1K_1,x^{n+s-m}))+m\}\\
\leq&\max\{\mathrm{reg}((I_1J_1K_1,x^s)),\mathrm{reg}((I_1J_1,x^{n-s}I_1K_1,x^{m-s}))+s,\\
&\mathrm{reg}((I_1J_1,I_1K_1,J_1K_1,x^{n+s-m}))+m\}\\
\leq&\max\{\mathrm{reg}((I_1J_1K_1,x^s)),\mathrm{reg}((I_1J_1,x^{n-s}))+s,\mathrm{reg}((I_1J_1,I_1K_1,x^{m-n}))+n,\\
&\mathrm{reg}((I_1J_1,I_1K_1,J_1K_1,x^{n+s-m}))+m\}.
\end{align*}
As the first two cases, one can show that $\mathrm{reg}((I_1J_1K_1,x^s)),\mathrm{reg}((I_1J_1,x^{n-s}))$ and
$\mathrm{reg}((I_1J_1,I_1K_1,x^{m-n}))+n$ are all not more than $\mathrm{reg}(I)+\mathrm{reg}(J)+\mathrm{reg}(K)$.
We only check the last term in the above equation. By Lemma 3.1,
we have
\begin{align*}
&\mathrm{reg}((I_1J_1,I_1K_1,J_1K_1,x^{n+s-m}))+m\\
=& \mathrm{reg}((I_1J_1,I_1K_1,J_1K_1))+n+s-1\\
\leq&\mathrm{reg}(I_1)+\mathrm{reg}(J_1)+\mathrm{reg}(K_1)+n+s-2\\
=& \mathrm{reg}(I_1)+\mathrm{reg}(J)+\mathrm{reg}(K)\leq \mathrm{reg}(I)+\mathrm{reg}(J)+\mathrm{reg}(K).
\end{align*}
For the term in (3.2), by Lemma 1.2,
\begin{align*}
&\mathrm{reg}((I_1,J_1K_1,x^{m-n}J_1,x^{m-s}K_1,x^m ))+n+s\\
\leq&\max\{\mathrm{reg}((I_1,J_1K_1,x^{m-n}))+n+s,\mathrm{reg}((I_1,J_1,x^{n-s}K_1,x^n))+m+s\}\\
\leq&\max\{\mathrm{reg}((I_1,J_1K_1,x^{m-n}))+n+s,\mathrm{reg}((I_1,J_1,x^{n-s}))+m+s,\\
&\mathrm{reg}((I_1,J_1,K_1,x^s))+m+n\}.
\end{align*}
We only look at the third term in the inequation above, the other two can be
checked similarly.
\begin{align*}
\mathrm{reg}((I_1,J_1,K_1,x^s))+m+n=&\mathrm{reg}((I_1,J_1,K_1))+m+n+s-1\\
\leq& \mathrm{reg}(I_1)+\mathrm{reg}(J_1)+\mathrm{reg}(K_1)+m+n+s-3\\
=& \mathrm{reg}(I)+\mathrm{reg}(J)+\mathrm{reg}(K).
\end{align*}
Now for the case that $m\leq n+s$, we have shown that $\mathrm{reg}(IJK)\leq
\mathrm{reg}(I)+\mathrm{reg}(J)+\mathrm{reg}(K)$. If $m> n+s$, one can get the conclusion
 by using the same method as above, and we omit the details here.
Now the theorem is proved.
\end{proof}


\vspace{0.5cm}



\vspace{0.2cm}


\begin{thebibliography}{10}

\bibitem[1]{cha} Chardin, M., Minh N. C., Trung, N. V. (2007) On the regularity of products and
intersections of complete intersections. Proc. Amer. Math. Soc. 135:1597-1606.
\bibitem[2]{Con} Conca, A., Herzog, J. (2003). Castelnuovo-Mumford regularity of products of ideals.
Collect. Math. 54:137-152.
\bibitem[3]{Her} Herzog, J. (2007). A generalization of the Taylor complex construction.
Comm. Algebra 35:1747-1756.
\bibitem[4]{Hoa} Hoa, L., Trung, N. V. (1998). On the Castelnuovo-Mumford regularity and the
arithmetic degree of monomial ideals. Math. Z. 229:519-537.
\bibitem[5]{Str} Sturmfels, B. (2000). Four counter examples in combinatorial algebraic geometry.
J. Algebra 230:282-294.
\bibitem[6]{Zho} Zhongming, T., Cheng, G. (2016) On the regularity of operations of ideals.
Comm. Algebra 44:2938-2944.

\end{thebibliography}
\end{document}